%% file: main.tex
\documentclass[12pt]{amsart}

\input{preamble.tex}
\begin{document}

\title[A decomposition for the Schr{\"o}dinger equation]{A decomposition for the Schr{\"o}dinger equation with applications to bilinear and multilinear
estimates}
\date{\today}
\author[F. Hern{\'a}ndez]{Felipe Hern{\'a}ndez}
\email{felipeh@mit.edu}
\keywords{Schr{\"o}dinger equation, bilinear Strichartz, multilinear restriction}
\subjclass[2010]{Primary 35Q41; Secondary 42B37}

\begin{abstract}
  A new decomposition for frequency-localized solutions to the Schrodinger
  equation is given which describes the evolution of the wavefunction using 
  a weighted sum of Lipschitz tubes. 
  As an application of this decomposition, we provide a new proof of the 
  bilinear Strichartz estimate as well as the 
  multilinear restriction theorem for the paraboloid.
\end{abstract}

\maketitle

This paper introduces new way of decomposing solutions to the 
Schr{\"o}dinger equation, which approximates the evolution of 
the mass distribution of a solution as a sum of spacetime Lipschitz tubes.  
There are several ways in which this decomposition differs from existing 
decompositions of solutions to the Schr{\"o}dinger equation, such as the 
wavepacket decomposition. First, whereas the wavepacket decomposition breaks 
the solution up into tubes of width $R^{1/2}$ and length $R$, the tubes 
appearing in this decomposition have unit width and can be made arbitrarily 
long.  Second, the tubes involved are curved instead of straight.  Finally, the 
decomposition involves no cancellation.  Whereas in the usage of the wavepacket 
decomposition it is often necessary to induct on scales in order to resolve 
possible cancellations between tubes, all pieces of the decomposition 
introduced here are positive so there is no opportunity for cancellation.  The 
lack of cancellation to exploit and the relatively loose control of the shape 
of the tubes means that this decomposition seems unable to prove many dispersive
inequalities used in the literature on the Schrodinger equation.  However,
the decomposition is useful enough to prove the bilinear Strichartz estimate
and a special case of the multilinear restriction estimate.  One motivation
for this work was to find proofs of these estimates that can be generalized
to non-Euclidean settings such as hyperbolic space. 

Now let us set up the general situation for our results.
Let $u:\Real^d\times\Real\to\Complex$ be a solution to the Schr{\"o}dinger 
equation
\begin{equation}
  i\partial_t u + \Delta u= 0
  \label{schrod-eq}
\end{equation}
with initial condition $u(x,0) = u_0(x)\in L^2(\Real^d)$.  For convenience
we will write $u_t$ for the function $u_t(x) = u(t,x)$.  If $u_t$ is localized
in frequency space, so that the Fourier transform $\supp \hat{u_t} \subset B_1$ 
is contained in the unit ball, then physical intuition about the 
Schr{\"o}dinger equation suggests that $u_t$ has two important properties.
First, because of the uncertainty principle, one expects that $u_t$ (and thus
$|u_t|^2$) does not vary much (is ``locally constant'') on unit scales.  
Moreover, because $u_t$ has bounded momentum, one expects that $u_t$ enjoys
something akin to finite speed of propagation.  Together these suggest that 
it is possible to describe the evolution of the mass $|u_t|^2$ in terms of 
discrete ``packets'' that travel along finite-speed paths.  The point of
this paper is to make this intuition precise.  

To set up some more notation, suppose
$\gamma:\Real\to\Real^d$ is some differentiable path which will represent the
motion of one such packet.  We say that $\gamma$ has speed at most
$V$ if $|\nabla \gamma| \leq V$.  Moreover we define the tube of width $r$
centered at $\gamma$ by 
\[
  T_{\gamma,r} := \{(x,t)\in\Real^d\times\Real; |x-\gamma(t)|\leq r\}.
\]
We will use $T_{\gamma,r}$ and its indicator function interchangeably.  Then
our main result is the following decomposition.
\begin{thm}[Skinny Lipschitz Tube decomposition]
  Let $u_t$ solve the Schr{\"o}dinger equation~\eqref{schrod-eq}
  with frequency localization $\supp \hat{u_0} \subset B_1$.  There
  exists a radius $r$, a speed limit $V$, and a constant $C$ 
  with the following property: For any $R>0$,
  there exists a countable collection of paths $\{\gamma_i\}$ with
  speed at most $V$ and weights $\{w_i\}$ such that 
  \begin{equation}
    |u_t(x)|^2 \leq \sum_{i} w_i T_{\gamma,r}(x,t)
    \label{ptwise-bd}
  \end{equation}
  for any $(x,t)\in\Real^d\times[-R,R]$, and moreover this bound is 
  efficient in the sense that 
  \begin{equation}
    \sum_i r^d w_i \leq C \int_{\Real^d} |u_t|^2\,dx.
    \label{L1-efficient-bd}
  \end{equation}
  \label{tube-decomp-thm}
\end{thm}

\emph{Remark:} By rescaling the weights $w_i$ so that $\sum_i w_i = 1$, it
is possible to think of them as probabilities.  In this interpretation, $w_i$
is the probability that the particle represented by the wavefunction $u_t$
takes the path described by $\gamma$.  However, since the paths are far from
unique, this is not a very accurate description.

The proof of this theorem has two main ingredients.  The first ingredient
is a pair of inequalities which demonstrate two key properties 
of $u_t$.  The first inequality says that $u_t$
is \emph{locally constant}, in the sense that $|u_t(x)|^2$ can be bounded
pointwise in terms of a weighted average centered at $x$.  The second
says that $|u_t|^2$ has \emph{finite speed}.  The statement of the finite
speed property is slightly more complex, but suffice it to say that this is
where we encode the ``bounded momentum'' intuition.  
The second ingredient of the proof is combinatorial.  We state and 
prove an analogue of Theorem~\ref{tube-decomp-thm} in a discrete situataion.
The main idea here is to construct discrete paths
one step at a time.  The existence of these one-step paths is guaranteed
by the Max-Flow Min-Cut theorem.  

By itself, Theorem~\ref{tube-decomp-thm} is unsuitable for applications 
which require a variety of frequency localizations.  This is easily 
fixed with an application of scaling and Galilean symmetries.
\begin{cor}
  Let $u_t$ be a solution to the Schr{\"o}dinger equation with frequency
  localization
  \[
    \supp \hat{u_0} \subset B_{\rho}(\xi),
  \]
  where $\xi\in\Real^d$ is some average momentum, and $\rho$ is the 
  frequency uncertainty.  Let $r,V,C>0$ be the same constants given
  in Theorem~\ref{tube-decomp-thm}, and let $T>0$ be arbitrary.  
  Then there exists a set of paths
  $\{\gamma_i\}_{i=1}^\infty$ with $|\nabla \gamma_i - \xi|< V\rho$, 
  and weights $\{w_i\}_{i=1}^\infty$ such that  
  \[
    |u_t(x)|^2 \leq \sum_i w_i T_{\gamma_i, r\rho^{-1}}(x,t)
  \]
  and
  \begin{equation}
    \sum_{i} (r\rho^{-1})^d w_i \leq C \int_{\Real^d} |u_t(x)|^2\,dx.
    \label{scaled-L1-bd}
  \end{equation}
  \label{scaled-decomp-cor}
\end{cor}
\begin{proof}[Proof of Corollary~\ref{scaled-decomp-cor} using 
  Theorem~\ref{tube-decomp-thm}]
  Consider the function 
  \[
    u'(x,t) = e^{i(|\xi|^2t/\rho^2 - \xi\cdot x/\rho)}
    u(x\rho^{-1} + t\xi\rho^{-2}, t\rho^{-2}),
  \]
  which is still a solution to the Schr{\"o}dinger equation (by Galilean and
  scale invariance), and which moreover satisfies the frequency
  localization $\supp \hat{u'} \subset B_1$.  Apply 
  Theorem~\ref{tube-decomp-thm} to $u'$ and then apply the scaling and 
  Galilean symmetries again in reverse.
\end{proof}

As mentioned above, Theorem~\ref{tube-decomp-thm} and 
Corollary~\ref{scaled-decomp-cor} seem unsuitable for some applications.
In particular, one apparent obstacle is that Theorem~\ref{tube-decomp-thm}
provides no guarantee about the \emph{dispersive} behavior of the tubes.  It
may be that there is a single tube which remains coherent for all time.  Thus,
Theorem~\ref{tube-decomp-thm} cannot by itself prove something such as the
Strichartz inequality.  However, there are examples of useful estimates in
which dispersion does not play as large a role.  The two we illustrate here
are the bilinear Strichartz estimate and the multilinear restriction theorem.

\subsection{Bilinear Strichartz}
The bilinear Strichartz estimate was introduced by Bourgain 
in~\cite{bourgain98}.  It handles the interaction between high-frequency and 
low-frequency solutions to the Schr{\"o}dinger equation.  
To state it, let
\[
  A_N := \{\xi; N/2\leq |\xi| \leq 2N\}
\]
denote the annulus at scale $N$.  

\begin{thm}[Bilinear Strichartz~\cite{bourgain98}]
  Let $M\ll N$, and let $u_0,v_0\in L^2(\Real^d)$ have the frequency 
  localizations $\supp \hat{u}\subset A_N$, $\supp \hat{v}\subset A_M$.  Then
  \begin{equation}
    \|u_t v_t\|_{L^2_{x,t}(\Real^d\times\Real)}
    \leq C \frac{M^{(d-1)/2}}{N^{1/2}} \|u_0\|_{L^2_x(\Real^d)}
    \|v_0\|_{L^2_x(\Real^d)}.
    \label{bilinear-Strich-bd}
  \end{equation}
  \label{bi-Strich-thm}
\end{thm}

This estimate is one of the ingredients in the I-method, 
which has been used to prove global well-posedness
for the nonlinear Schr{\"o}dinger equation for rough data\cite{CKSTT02}.  
The first 
proof of the bilinear Strichartz estimate relied heavily on properties of the 
Fourier transform which do not easily generalize to non-Euclidean settings.
Since the first proof by Bourgain, Tao has given a more physical-space proof
in~\cite{tao2010}.  Moreover it seems that the wavepacket approach to the 
bilinear Strichartz estimate for the wave equation described 
in~\cite{KRT02} might work for the Schr{\"o}dinger equation as well.  But 
to my knowledge,
the only proof of the bilinear Strichartz estimate on a manifold is due to 
Hani~\cite{hani2012}. Hani used this estimate and the I-method to prove a 
global well-posedness estimate for the cubic nonlinear Schr{\"o}dinger equation
in~\cite{hani2012global}. The approach in this paper to the bilinear 
Strichartz estimate is quite different from any mentioned above, and I hope 
that it may be extended to non-Euclidean settings such as hyperbolic space.  

\subsection{Multilinear Restriction}
The multilinear restriction estimate was first stated by Bennett, Carbery,
and Tao~\cite{BCT06}.  It is a version of the restriction problem which is 
simplified to handle only quantitatively transverse interactions.  Here
we work with a very special case, since we only want to show how the 
decomposition of Theorem~\ref{tube-decomp-thm} can be used.

\begin{thm}[Multilinear restriction~\cite{BCT06}]
  Let $u_i$ be solutions to the Schrodinger equation for $0\le i\le d$.\footnote{In the statement of this theorem we drop the notation that $u_t(x) = u(t,x)$.}
  Let $r=(10d)^{-1}$ and suppose that $u_0$ has 
  frequency support $\supp \hat{u_0} \subset B_r(0)$
  and the $u_i$ have frequency support $\supp \hat{u_i}\subset B_r(e_i)$,
  where $e_i$ is any orthonormal basis for $\Real^d$.  Then for every
  $\eps>0$ there is a constant $C_\eps>0$ such that 
  \begin{equation}
    \int_{|x|<R}\int_{|t|<R}
    \prod_{i=0}^d |u_i(t,x)|^{2/d} \le 
    C_\eps R^\eps \prod_{i=0}^d \|u_i\|_{L^2_x}^{2/d}
    \label{multi-rest-bd}
  \end{equation}
  for any $R>0$.
  \label{multi-rest-thm}
\end{thm}

This kind of estimate has been used to make progress on the restriction 
problem~\cite{BG11} and
more recently to prove the $\ell^2$ decoupling conjecture~\cite{BD14}.  
We remark
that the main difficulty in proving the multilinear restriction is an
incidence geometry problem about tubes (see~\cite{guth2015short} for a short 
discussion of the problem), and our only contribution is to provide a new
derivation of the restriction estimate from the geometric problem.  In
particular this derivation avoids the use of induction on scales.

\subsection{Plan of the Paper}
The paper is organized as follows.  Section~\ref{ests-section} contains a 
derivation of the locally constant and finite speed properties to frequency
localized solutions.  Section~\ref{discrete-decomp-section} proves a
discrete analogue of the decomposition.  In Section~\ref{decomp-section},
the full proof of Theorem~\ref{tube-decomp-thm} is given.
Section~\ref{applications-section} proves the bilinear Strichartz estimate,
Theorem~\ref{bi-Strich-thm}, and the multilinear restriction estimate, 
Theorem~\ref{multi-rest-thm}, using
the tube decomposition in Corollary~\ref{scaled-decomp-cor}.

\subsection*{Acknowledgements}
I would like to thank Larry Guth for useful comments.
I also thank Vincent Tjeng for 
showing me how to use the Max-Flow Min-Cut theorem to prove 
Lemma~\ref{one-layer-lemma}.

\input{kernel_ests.tex}
\input{discrete.tex}
\input{bilinear.tex}

\bibliographystyle{amsalpha}
\bibliography{./refs}
\end{document}

%% file: preamble.tex
\usepackage{amsmath, amsfonts, amssymb, amsthm, enumerate}
\usepackage{mathrsfs}
\usepackage{hyperref}
\usepackage{gensymb}
\usepackage{graphicx, pstricks}

\makeatletter
\def\input@path{{../figures/}{../bibliography/}}
\makeatother

\newcommand{\eps}{{\varepsilon}}
\newcommand{\Real}{{\mathbb{R}}}
\newcommand{\Complex}{{\mathbb{C}}}
\newcommand{\Integer}{{\mathbb{Z}}}
\newcommand{\Cube}{{\mathcal{Q}}}

\newcommand{\Schwartz}{{\mathcal{S}}}
\newcommand{\Prob}{{\mathbb{P}}}
\DeclareMathOperator{\dist}{dist}
\DeclareMathOperator{\supp}{supp}

\DeclareMathOperator{\Capacity}{cap}
\DeclareMathOperator{\val}{val}
\DeclareMathOperator{\Imgnry}{Im}

\newtheorem{thm}{Theorem}
\newtheorem{cor}{Corollary}
\newtheorem{lemma}{Lemma}
    
\newtheorem{propn}{Proposition}

\newtheorem{question}{Question}

%% file: kernel_ests.tex
\section{Properties of Frequency-Localized solutions}
\label{ests-section}
In this section we prove variants of the locally constant (LC) and 
finite speed (FS) properties for solutions $u_t$ to the Schrodinger equation
satisfying 
\begin{equation}
  \supp\hat{u_0}\subset B_1
  \label{freq-local-eq}.
\end{equation}
\subsection{Locally Constant}
We begin with the locally constant property.  We say that a bounded
function $f\in L^\infty(\Real^d)$ is 
\emph{localized to the unit ball} if there exists some constant
$C$ such that $|f(x)| \leq C$ and 
\[
  |f(x)|\leq C|x|^{-10d}, \quad |\nabla f(x)| \leq C|x|^{-10d-1}.
\]
\begin{propn}
  There exists some positive function $\mu$ which is localized to the 
  unit ball and such that for any 
  function $u_0$ satsfying the frequency~\eqref{freq-local-eq}, 
  \begin{equation}
    \sup_{y\in B_1(x)} |u_0(y)|^2 \leq C\int |u_0(y)|^2 \mu(x-y)\,dy.
    \tag{LC}
    \label{LC-bd}
  \end{equation}
  Moreover, $\mu$ can be chosen such that the translates of $\mu$ form 
  a partition of unity,
  \begin{equation}
    \sum_{a\in \Integer^d} \mu(x-a) = 1,
    \label{Z-po1-eq}
  \end{equation}
  and such that 
  \begin{equation}
    c |x|^{-10d} \leq \mu.
    \label{poly-decay-bd}
  \end{equation}
  \label{LC-propn}
\end{propn}
\emph{Remark:} The extra conditions~\eqref{Z-po1-eq} and~\eqref{poly-decay-bd}
can be tacked on by starting with some $\mu''$ localized to the unit ball
and then making it larger.  At first glance this seems quite lossy; the point
however is that in the proof we will use the bound~\eqref{LC-bd} only once,
and the extra mass in $\mu$ will allow us to arrive at a 
finite speed property (in the next subsection) which has no loss of constants, 
and which will be iterated many times.

\emph{Remark:} The bound in~\eqref{poly-decay-bd} could be tweaked in several
ways.  For the purposes of the next section, it would suffice to know that
$\mu$ is smooth (approximately locally constant on unit scales) and that 
\[
  \int_{E} |\partial_i \mu| \leq C \int_{\partial E} \mu
\]
whenever $E\subset \Real^d$ satisfies $E\cap B_1 = \emptyset$.  This is
a more general condition which may be useful in non-Euclidean settings, 
but the easiest way I know to verify it is to enforce
that $\mu$ has polynomial decay, hence the bound~\eqref{poly-decay-bd}. 

\begin{proof}
Let $\chi\in L^2(\Real^d)$ satisfy $\hat{\chi}(\xi) = 1$ for $|\xi|\leq 1$ and
$\hat{\chi}(\xi) = 0$ outside $|\xi|\geq 2$, with a smooth cutoff in between.
By choosing the function appropriately, we can
enforce that $\chi$ localized to the unit ball.
The Fourier localization of $u_0$ yields $u_0 = \chi\ast u_0$.  This allows
us to bound, using Cauchy-Schwartz,
\begin{align*}
  |u(x)|^2 = |\chi\ast u(x)|^2 
  &= \left| \int u(y) \chi(x-y)\,dy \right|^2
  \\&\leq  \left(\int |\chi(x-y)|\,dy\right) \int |u(y)|^2 |\chi(x-y)|\,dy 
  \\&= c_d \int |u(y)|^2 |\chi(x-y)|\,dy.
\end{align*}
Now define the function $\mu''$ by
\[
  \mu''(y) = \sup_{|z-y|\leq 1} |\chi(z)|.
\]
The fact that $\chi$ is localized to the unit ball implies that $\mu''$ is also,
with a larger constant.  We now modify $\mu''$ to satisfy~\eqref{Z-po1-eq}
and~\eqref{poly-decay-bd}.  As long as these modifications only increase $\mu'$,we can ensure that~\eqref{LC-bd} is still satisfied.

First, since $\mu''$ is localized to the unit ball, we can find some function
$\mu'$ which is smooth, is localized to the unit ball, has polynomial
decay $\mu' \geq c|x|^{-10d}$, and satisfies $\mu'' \leq \mu'$.  Moreover,
each dilation $\mu'_R(x) = \mu'(x/R)$ is still localized to the unit ball
and satisfies~\eqref{poly-decay-bd} with different constants.  By making 
$R$ large enough the periodic function
\[
  p(x) = \sum_{a\in \Integer^d} \mu(x-a) 
\]
becomes very smooth, so that $\mu = \mu'(x) / p(x)$ which 
satisfies~\eqref{Z-po1-eq} still obeys all previous bounds.  
\end{proof}

\subsection{Finite Speed}
In this section we prove a version of the finite speed property for frequency
localized solutions to the Schrodinger equation.  We first 
establish some notation for the section.  Define the set
\begin{equation}
  H := \{h\in\Integer^d; \max_{1\leq i\leq d} |h_i| \leq 1\},
  \label{H-defn-eq}
\end{equation}
which has cardinality $3^d$ and is the cube centered at the origin.  Given
any set $A\subset\Integer^d$, we let $A+H$ denote the sumset
\[
  A+H := \{a+h\in\Integer^d; a\in A, h\in H\}.
\]
In other words, $A+H$ consists of all lattice points in $\Integer^d$ which are
within a distance $1$ of $A$ in the $\ell^\infty$ norm.  We also define
$\partial A = (A+H) \setminus A$, the lattice points that are within $1$ of $A$
but are not in $A$.  

We will also be using $\mu$ and its translates from Proposition~\ref{LC-propn}.
For convenience, let $\mu_a(x) = \mu(x-a)$ be the function $\mu$ centered
at the lattice point $a\in\Integer^d$.  Moreover, for any set 
$B\subset\Integer^d$, we write $\mu_B = \sum_{b\in B} \mu_b$.
\begin{propn}
  Let $u_0$ have frequency localization as in~\eqref{freq-local-eq}
  and $u_t = e^{it\Delta}u_0$ be the free Schrodinger evolution of $u_0$.
  Moreover let $\mu$ satisfy the conclusions of Proposition~\ref{LC-propn}.

  Then there exists some time $\tau>0$ and such that for any 
  $A\subset\Integer^d$ and $0<t<\tau$,
  \begin{align}
    \begin{split}
      \int |u_t|^2 \mu_A &\leq \int |u_0|^2 \mu_{A+H}, \\ 
      \int |u_0|^2 \mu_A &\leq \int |u_t|^2 \mu_{A+H}.
      \label{FS-bd}
    \end{split}
      \tag{FS}
  \end{align}
  \label{FS-propn}
\end{propn}
\emph{Remark:} By time reversal symmetry, the equations in~\eqref{FS-bd} are
equivalent, so we only focus on proving the first.  

We will need a few lemmas in order to complete the proof of this proposition.
\begin{lemma}
  There exists a constant $C>0$ such that for any $A\subset\Integer^d$
  and any $1\leq j\leq d$,
  \begin{equation}
    |\partial_j \mu_A| \leq C |\mu_{\partial A}|.
    \label{bd-A-bd}
  \end{equation}
  \label{bd-A-lemma}
\end{lemma}
\begin{proof}
  Let $x\in \Real^d$, and let $\dist(x,\partial A)$ denote the minimum distance
  between $x$ and a point in $\partial A$.  We split the analysis into
  cases.  The first case occurs when $x$ is near the boundary of $A$, 
  so that $\dist(x,\partial A) < 100d$.  Then
  $\mu_{\partial A}(x) = \Omega_d(1)$, and $|\partial_j \mu_A| = O_d(1)$, 
  so the bound holds for some constant.

  The second case is when $x$ is in the ``exterior'' of $A$,
  so $\dist(x,A)>100d$.  Then we write, 
  expanding the definition of $\mu_A$, applying the triangle 
  inequality, and using the fact that $\mu$ is localized to the unit ball
  \[
    |\partial_j \mu_A| \le C\sum_{a\in A} |x-a|^{-10d-1}.
  \]
  We would like to turn this sum into an integral.  Let 
  $\Cube(A)\subset\Real^d$ denote the union of unit cubes centered at the 
  points of $A$.  Since we are well away from the origin, the function
  $|y|^{-10d-1}$ doesn't change much on unit scales, so 
  \[
    \sum_{a\in A} |x-a|^{-10d-1}
    \leq C \int_{\Cube(A)} |x-a|^{-10d-1}\,da.
  \]
  Now observe that the integrand can be written as 
  $-9d\partial_i (x_i |x|^{-10d-1})$, so by integrating by parts we obtain
  \[
  \int_{\Cube(A)} |x-a|^{-10d-1}\,da
    \leq c\int_{\partial \Cube(A)} |x-a|^{-10d}\,d\sigma(a),
  \]
  where $d\sigma$ denotes the surface measure on the exposed faces
  of the cubes.  We can assign each face of an exposed cube to a point in 
  $\partial A$, with each point in $\partial A$ being chosen at most
  $2^d$ times.  Again because we are away from the origin, the function
  $|y|^{-10d}$ doesn't change much on unit scales,
  \[
    \int_{\partial \Cube(A)} |x-a|^{-10d}\,d\sigma(a)
    \leq C \sum_{a\in \partial A} |x-a|^{-10d}.
  \]
  Applying~\eqref{poly-decay-bd} and chaining the inequalities 
  we can conclude that 
  \[
    |\partial_j \mu_A| \le C |\mu_{\partial A}|.
  \]

  Finally we must handle the case that $x$ is in the ``interior'' of 
  $A$, so that $\dist(x,A) < 100d$.  This time the bound
  \[
    |\partial_j \mu_A| \leq  C\sum_{a\in A} |x-a|^{-10d-1}
  \]
  is much too lossy because of the cancellation involved.  To exploit this
  cancellation, we observe that because $\mu$ forms a partition of unity, 
  \[
    |\partial_j \mu_A| = |\partial_j \mu_{A^c}|,
  \]
  where $A^c = \Integer^d\setminus A$ is the complement of $A$.  Now 
  $\dist(x, A^c) > 100d$, so apply the same argument as above.
\end{proof}

\begin{lemma}
  Let $K\in\Schwartz(\Real^d)$ be a Schwartz function (smooth and rapidly
  decaying).  Then there exists $C>0$, independent of $K$,
  such that for any $B\subset\Integer^d$,
  \begin{equation}
    |K\ast \mu_B(x)|\leq C\left(\int |K(y)|(1+|y|^{10d})\,dy\right)\mu_B(x)
    \label{muB-conv-bd}
  \end{equation}
  \label{muB-conv-lemma}
\end{lemma}
\begin{proof}
  By the triangle inequality and translation invariance, 
  it suffices to show that 
  \[
    |K\ast \mu(x)| \leq C\mu(x).
  \]
  Now let $y\in \Real^d$.  We claim that
  \begin{equation}
    \mu(x-y) \leq C(1 + |y|^{10d}) \mu(x).
    \label{translate-mu-bd}
  \end{equation}
  Again we split the analysis into cases.  The first case 
  occurs when $|x|\leq 50|y|$ and $|y|>10$. 
  Then $\mu(x) \geq c|50y|^{-10d}$, so 
  \[
    \mu(x-y) \leq C \leq C |y|^{10d} \mu(x).
  \]
  The second case is $|x|\leq 50|y|$ and $|y|\le 10$. Then
  both $\mu(x)$ and $\mu(x-y)$ are $\Theta(1)$.  Finally, if $|x|>50|y|$
  then
  \[
    \mu(x-y) \sim \mu(x) 
  \]
  because both are $\Theta(|x|^{-10d})$.  This concludes the proof of
  the claim~\eqref{translate-mu-bd}.

  We use this claim to bound the convolution:
  \begin{align*}
    |K\ast \mu(x)| &= | \int K(y) \mu(x-y)\,dy | \\
    &\leq \int |K(y)| |\mu(x-y)|\,dy \\
    &\leq C \mu(x) \int |K(y)| (1+|y|^{10d})\,dy.
  \end{align*}
  Now because $K$ is rapidly decaying, the integral on the right is bounded
  by some constant so we are done.
\end{proof}

\begin{proof}[Proof of Proposition~\ref{FS-propn}]
  Subtract $|u_0|^2 \mu_A$ from both sides of~\eqref{FS-bd}, apply the 
  fundamental theorem of calculus, and use the local conservation of mass
  \[ 
    \partial_t |u_t|^2 = \partial_j \Imgnry (\bar{u_t} \partial_j u_t)
  \]
  to note that it suffices to show
  \[
    \int (|u_t|^2-|u_0|^2)\mu_A 
  = \int \left(\int_0^t \partial_s |u_s|^2\,ds \right)\mu_A
  = -\int
  \left(\int_0^t\partial_j\Imgnry(\overline{u_s}\partial_j u_s)\,ds
  \right) \mu_A
  \leq \int |u_0|^2 \mu_{\partial A}.
\]
We may integrate the LHS by parts.  Moreover we use
$u_s = e^{is\Delta} u = e^{is\Delta}(\chi\ast u) = (e^{is\Delta}\chi)\ast u$,
and define $K_s = e^{is\Delta}\chi$.  We can rewrite this inequality
as
\[
  \mathrm{Im}( \int_0^t \int 
  (\overline{K_s}\ast \overline{u_0})(\partial_j K_s\ast u_0) \partial_j \mu_A)
  \leq \int |u_0|^2 \mu_{\partial A}.
\]
We now need a few facts about $K_s$ and $\partial K_s$.  They are both 
concentrated in the ball $|x|\lesssim s$ and have rapidly decaying tails (the
estimate on the tails is a repeated application of integration by parts).  
In particular, the quantities 
\[
  \int |K_s(y)|(1+|y|^{10d})\,dy, \quad \int |\partial_j K_s(y)|(1+|y|^{10d})
  \,dy
\] 
are uniformly bounded near $s=0$.  We can therefore
apply Cauchy-Schwartz several times to bound the LHS by
\begin{align*}
  \mathrm{Im}( \int_0^t \int 
  (\overline{K_s}\ast \overline{u_0})(\partial_j K_s\ast u_0) \partial_j \mu)
  &\leq
  \int_0^t \int (|K_s|\ast |u_0|)(|\partial_j K_s|\ast|u_0|)|\partial_j\mu_A| \\
  &\leq C
  \int_0^t 
  \left(\int (|K_s|\ast |u_0|)^2|\partial_j\mu_A|\right)^{1/2}
  \left(\int (|\partial_j K_s|\ast|u_0|)^2 |\partial_j\mu_A|\right)^{1/2} \\
  &\leq C \tau
  \sup_{0<s<\tau, 1\le j\le d}\{\int |u_0|^2 |K_s|\ast|\partial_j \mu_A|, 
  \int |u_0|^2 |\partial_j K_s|\ast |\partial_j\mu_A|\}.
\end{align*}
By applying Lemma~\ref{bd-A-lemma} followed by Lemma~\ref{muB-conv-lemma},
we can bound the right hand side by
\[
  C\tau \int |u_0|^2 |\mu_{\partial A}|.
\]
Upon taking $\tau<1/C$ we are done.
\end{proof}

%% file: discrete.tex
\section{The Discrete Situation}
\label{discrete-decomp-section}
In the continuous situation, we have an evolving mass distribution $|u_t|^2$ 
and we would like to understand it in terms of packets moving around according
to Lipschitz paths.  This can be discretized in space and time to come up
with the following set-up.  We have a (possibly infinite) 
digraph\footnote{In our notation, a digraph has directed edges but it is 
allowed for both the edges $(u,v)$ and $(v,u)$ to exist.} $G = (V,E)$
with bounded degree (self-loops are allowed).  The nodes represent a
discretization of the space (in our case $V=\Integer^d$), and edges represent
possible movements of mass within a timestep (in our case $(a,b)\in\Integer^2$
is an edge if $\|a-b\|_{\ell^\infty} \leq 1$).  Instead of an 
evolving mass distribution $|u_t|^2$ there
is a sequence of mass distributions $w_i:V\to\Real^+$ for $i=1,2,\cdots, N$.  
We say a sequence of vertices $p = (p(1),\cdots, p(N))$ is a path of
length $N$ if $(p(i),p(i+1))$ is an edge for each $1\le i<N$.  The set 
of all such paths is denoted $\mathcal{P}_N$.  
We want to describe this evolution of mass as a sum of packets that each
move according to different paths.

\begin{question}
  Is it possible to find a weighting $\alpha:\mathcal{P}_N\to\Real^+$
  on the set of paths $\mathcal{P}_N$ of length $N$
  such that for all $1\le i\le N$, and all vertices $v\in V$,
  \[
    w_i(v) = \sum_{p\in \mathcal{P}; p(i)=v} \alpha(p)?
  \]
\end{question}

Sometimes it is impossible to come up with such a weighting.  For
a trivial example, there may be no edges in $E$ at all, and so there can be
no paths.  The only allowable mass distribution is the trivial one, 
$w_i = 0$.  A less trivial constraint is that $w_i$ must 
satisfy a conservation of mass.  Indeed, observe that  
\[
  \sum_{v\in V} w_i(v) = \sum_{v\in V} \sum_{p\in \mathcal{P}; p(i)=v}
  \alpha(p)
  = \sum_{p\in\mathcal{P}} \alpha(p)
\]
is independent of $i$.  In fact, $w_i$ must satisfy a \emph{local}
conservation law.  To state it, we make a few definitions.
For $A\subset V$ define $N^+(A)$, 
the outgoing neighborhood of $A$, to be the set
\[
  N^+(A) := \{v\in V; (a,v)\in E \text{ for some } a\in A\}.
\]
Similarly define the incoming neighborhood,
\[
  N^-(A) := \{v\in V; (v,a)\in E \text{ for some } a\in A\}.
\]
By an abuse of notation we will write $N^{\pm}(v)$ instead
of $N^{\pm}(\{v\})$ when $A$ consists of a single vertex $v$.  

The key property here is that if $p$ is a path and $p(i)\in A$, then 
$p(i+1)\in N^+(A)$ (and $p(i-1)\in N^-(A)$).  Thus,
\[
  \sum_{a\in A} w_i(a) = \sum_{p(i)\in A} \alpha(p)
  \leq \sum_{p(i+1)\in N^+(A)} \alpha(p) = \sum_{b\in N^+(A)} w_i(b).
\]
Likewise
\[
  \sum_{a\in A} w_i(a) \leq \sum_{b\in N^-(A)} w_{i-1}(b).
\]
The answer to our question above is that these local conservation laws are
not only necessary but also sufficient.

\begin{propn}
  Let $G=(V,E)$ as above, and let $w_i:V\to\Real^+$ be a sequence of positive
  weights on $V$.  Suppose that for every $A\subset V$, and $1\le i<N$
  the local conservation law holds,
  \begin{equation}
    \sum_{a\in A} w_i(a) \leq \sum_{b\in N^+(A)} w_{i+1}(b),
    \label{local-cons-bd}
  \end{equation}
  and in addition the global conservation law holds,
  \begin{equation}
    \sum_{v\in V} w_{1}(v) = \sum_{v\in V} w_{N}(v).
    \label{global-cons-bd}
  \end{equation}
  Then there exists a weighting $\alpha:\mathcal{P}_N\to\Real^+$ on the set
  of paths of length $N$ such that for every $1\le i\le N$ and every
  $v\in V$,
  \begin{equation}
    w_i(v) = \sum_{p(i) = v} \alpha(p).
    \label{discrete-decomp-eq}
  \end{equation}
  \label{discrete-decomp-propn}
\end{propn}

First we prove the Proposition in the case $N=2$. 
Then we deduce the case of general $N$ from the $N=2$ case.

\subsection{The case $N=2$}
When $N=2$, Proposition~\ref{discrete-decomp-propn} is a consequence of 
the Max-Flow Min-Cut (MFMC) theorem.  An introduction to the MFMC theorem
can be found online, for example see~\cite{Joseph07, Staples-Moore}. 
The theorem is originally due to Ford and Fulkerson~\cite{ford1956maximal}.  
Usually this theorem is stated in the
case of a finite network, but here our vertex set $V$ is infinite.  In general
this can be problematic, but the weights we consider are absolutely 
summable so there is no real difficulty.  

Now we set up notation so that we may state the MFMC theorem.  Let $U$ be a 
set of nodes with two distinguished elements, $s$ and $t$, which represent
the source and sink respectively. Let $c:U\times U\to\Real^+$.  We say $c$ is a
\emph{capacity} function if $c(u,v)>0$ implies $c(v,u) = 0$.  That is,
if we think of $c$ as being the capacity of a network of pipes connecting
the nodes of $U$, the pipes must only go in one direction.  A \emph{cut} 
is a subset $S\subset U$ such that $s\in S$ and $t\not\in S$.  
Given a capacity function $c$,
we define the \emph{capacity} of the cut $S$, written $\Capacity S$, by
\[
  \Capacity S = \sum_{u\in S, v\not\in S} c(u,v).
\]
The dual object of a cut is a flow.  A \emph{flow} is a function 
$f:U\times U\to\Real^+$ satisfying the Kirchoff's laws
\[
  f_{in}(v) := \sum_{u\in U} f(u,v) = \sum_{u\in U} f(v,u) =: f_{out}(v)
\]
for all $v\in U\setminus\{s,t\}$.  The \emph{value} of a flow is given 
by 
\[
  \val f = \sum_{u\in U} f(s,u).
\]
Because $f$ satisfies Kirchoff's laws except on $s$ and $t$, we may also write
\[
  \val f = \sum_{u\in U} f(u,t).
\]
The MFMC theorem relates the maximum value of a flow to the minimum capacity
of a cut.
\begin{thm}[Max-Flow Min-Cut~\cite{ford1956maximal}]
  Let $U$ be a finite vertex set with source $s\in U$ and sink $t\in U$,
  and let $c:U\times U\to\Real^+$ be a capacity on $U$.
  Then there exists a flow $f:U\times U\to\Real^+$ with $f\leq c$ such that 
  \[
    \val f = \min_{S\subset U} \Capacity S,
  \]
  where the minimum ranges over cuts $S\subset U$ with $s\in S$ and 
  $t\not\in S$.
\end{thm}

Notice that flows, which are weights on edges, can also be thought of weights
on the set of paths of length $2$.  This provides the main connection
between MFMC and the $N=2$ case of Proposition~\ref{discrete-decomp-propn}.
To illustrate this connection better, we set forth some definitions.
Let $V_1$ and $V_2$ be two identical copies of $V$, and let
$U=\{s,t\}\cup V_1\cup V_2$.  Let $c:U\times U\to\Real^+$ be the capacity
function defined by
\begin{equation}
  c(u,u') = 
  \begin{cases}
    w_1(u') &\mbox{if } u = s, u'\in V_1 \\
    w_1(u)  &\mbox{if } u \in V_1, u'\in N^+(u) \\
    w_2(u)  &\mbox{if } u \in V_2, u' = t \\
    0       &\mbox{else.}
  \end{cases}
  \label{cap-defn-eq}
\end{equation}
If we could apply MFMC directly to $U$ with capacity function $c$, then the 
local conservation laws~\eqref{local-cons-bd} and~\eqref{global-cons-bd} would
ensure that there is a suitable flow.  

Unfortunately $U$ as described may be 
infinite.  We therefore need to apply MFMC to finite subnetworks of $U$.  
Let $A\subset V$ be a finite subset of $V$, and let 
$U_A = \{s,t\}\cup A_1 \cup N^+(A_1) \subset U$, where 
we think of $A_1\subset V_1$ and $N^+(A_1)\subset V_2$.  We can also 
define the capacity $c_A = c|_{U_A}$, which is simply the restriction of 
$c$ to $U_A$.  We can apply MFMC to the network described by the nodes $U_A$
and capacity $c_A$ to obtain the following result.
\begin{lemma}
  Let $A\subset V$, and let $U_A$ and $c_A$ be as described above.  Then there
  exists a flow $f_A:U_A\times U_A\to\Real^+$ such that $f_A\leq c_A$ and
  \[
    f_A(s,u) = w_1(u) 
  \]
  for every $u\in A$. 
  \label{MFMC-lemma}
\end{lemma}
\begin{proof}
  Consider the flow network on $U_A$ with capacity $c_A$.  Let $S\subset U_A$
  be a cut.  We can write $S=\{s\}\cup S_1\cup S_2$,
  where $S_1\subset A_1$ and $S_2\subset N^+(A)$.  Now we write $\Capacity S$
  using the definition~\eqref{cap-defn-eq} of $c$
  \begin{align*}
    \Capacity S &= \sum_{u\in A_1\setminus S_1} c(s,u)
    + \sum_{\substack{u\in S_1 \\ v\in N^+(S_1)\setminus S_2}} c(u,v)
    + \sum_{v\in S_2} c(v,t) \\
    &= \sum_{u\in A_1\setminus S_1} w_1(u)
    + \sum_{\substack{u\in S_1 \\ c\in N^+(S_1)\setminus S_2}} w_1(u)
    + \sum_{v\in S_2} w_2(v).
  \end{align*}
  Define a subset $S'_1\subset S_1$ by
  \[
    S'_1 := \{u\in S_1; N^+(u) \subset S_2\}.
  \]
  If $u\in S_1\setminus S'_1$, then there is some $v\in N^+(u)$ such that 
  $v\not\in S_2$, so that $c(u,v) = w_1(u)$ contributes $w_1(u)$ to the 
  sum in $\cap S$.  Moreover we have $S_2 \subset N^+(S'_1)$, so 
  \begin{align*}
    \Capacity S &\geq \sum_{u\in A_1\setminus S'_1} w_1(u)
    + \sum_{v\in N^+(S'_1)} w_2(v) \\
    &\geq \sum_{u\in A_1\setminus S'_1} w_1(u)
    + \sum_{u\in S'_1} w_1(u) = \sum_{u\in A_1} w_1(u)
  \end{align*}
  where we have used the local conservation law~\eqref{local-cons-bd} with 
  $i=1$.  The expression on the RHS is exactly $\Capacity(\{s\})$, so by 
  MFMC there exists a flow $f_A$ on $U_A$ with $f_A\leq c_A$ such that 
  \[
    \val f_A = \sum_{u\in A} w_1(u).
  \]
  On the other hand since $f_A(s,u)\leq c(s,u) = w_1(u)$,
  \[
    \val f_A = \sum_{u\in A} f_A(s,u) \leq \sum_{u\in A} w_1(u).
  \]
  Thus $f_A(s,u) = w_1(u)$ for all $u\in A$, as desired.
\end{proof}

A compactness and limiting argument allows us to make a conclusion about the
entire network $U$.  Formulated in slightly different notation, we obtain
the $N=2$ case of Proposition~\ref{discrete-decomp-propn}.
\begin{lemma}
  Let $G=(V,E)$ be as in Proposition~\ref{discrete-decomp-propn},
  and let $w_1$ and $w_2$ be two weights on $V$ satisfying~\eqref{local-cons-bd}
  with $i=1$ and~\eqref{global-cons-bd} for $N=2$. 
  Then there exists a function $f:E\to \Real^+$ on the edges such that 
  \[
    w_1(v) = \sum_{u\in N^+(v)} f(v,u)
  \]
  and
  \[
    w_2(v) = \sum_{u\in N^-(v)} f(u,v).
  \]
  \label{one-layer-lemma}
\end{lemma}
\begin{proof}
  Let $A_k\subset V$ be an increasing sequence of sets, 
  $A_1\subset A_2\subset\cdots$, such that for every $v\in V$ there is 
  some $k$ such that $v\in A_k$.  Let $U_A$ be the network described above.
  For each $k$, Lemma~\ref{MFMC-lemma} provides some flow $f_k$ such that 
  $f_k\leq c$ and $f_k(s,u) = w_1(u)$ when $u\in V$.  We show that, up to 
  subsequence, $f_k$ converges to a flow $f$ on $U_A$ in $\ell^1$.  This is 
  done by a compactness argument.
  
  Indeed, let $\eps>0$.  We will prove the existence of a subsequence 
  $k(1),k(2),\cdots$ such that 
  \begin{equation}
    \sum_{(u,u')\in U\times U} |f_{k(i)} - f_{k(j)}| < \eps
    \label{subseq-eps-bd}
  \end{equation}
  for all $i,j\geq 0$.  First choose $K$ so large that 
  \[
    \sum_{u\in V\setminus A_K} w_1(u) \leq \eps,
  \]
  which is possible since $w_1$ is absolutely summable and $A_k$ is an 
  increasing sequence which converges to $V$.  

  There are finitely many pairs $(u,u')\in U_{A_K}\times U_{A_K}$, so by
  compactness we can find a sequence $k(1),k(2),\cdots)$ such that 
  \begin{equation}
    \sum_{(u,u')\in U_{A_K}\times U_{A_K}}
    |f_{k(i)}(u,u') - f_{k(j)}(u,u')| \leq \eps
    \label{finite-compact-bd}
  \end{equation}
  for all $i,j$.  Now we write
  \begin{align*}
    \sum_{(u,u')\in U\times U} |f_{k(i)}(u,u') - f_{k(j)}(u,u')| =
    &\sum_{u\not\in A_K} |f_{k(i)}(s,u) - f_{k(j)}(s,u)| 
    \\+&\sum_{\substack{u\not\in A_K \\ v\in V_2}} 
    |f_{k(i)}(u,v) - f_{k(i)}(u,v)| 
    \\+&\sum_{v\not\in N^+(A_K)} |f_{k(i)}(v,t)-f_{k(i)}(v,t)|
    \\+&\sum_{\substack{u\in A_K \\ v\in V_2}}
    |f_{k(i)}(u,v)-f_{k(j)}(u,v)|.
  \end{align*}
  We bound each term separately.  Because of~\eqref{finite-compact-bd}
  the last term is bounded by $\eps$.  The first term is bounded by $2\eps$
  due to the triangle inequality, the fact that $f\leq c$,
  and the definition of $A_K$.  The next two terms are actually identical
  to the first, once you apply the triangle inequality and the Kirchoff's 
  laws for $f_{k(i)}$ and $f_{k(j)}$.  We have therefore 
  shown~\eqref{subseq-eps-bd}.  
  
  A standard diagonalization argument now
  shows the existence of a subsequence, which we will not relabel explicitly,
  which is Cauchy in $\ell^1$.  More precisely, this means for every
  $\eps>0$ there exists $M$ such that when $n,m>M$,
  \[
    \sum_{(u,u')\in U\times U} |f_n(u,u')-f_m(u,u')| < \eps.
  \]
  By the completeness of $\ell^1$, this
  converges in $\ell^1$ to some function $f:U\times U\to\Real^+$.  This 
  convergence is strong enough to ensure that $f$ is still a flow, and 
  moreover $f(s,u) = w_1(u)$ for all $u\in V$.  This, combined with the 
  Kirchoff's laws, implies that 
  \[
    w_1(u) = \sum_{v\in N^+(u)} f(u,v)
  \]
  for all $u\in V$.  Moreover, Kirchoff's laws and the global
  conservation law~\eqref{global-cons-bd} ensure
  \[
    \sum_{u\in V} w_1(u) = \val f = \sum_{v\in V_2} f(v,t)
    \leq \sum_{v\in V} w_2(v) = \sum_{u\in V} w_1(v).
  \]
  Since $f(v,t)\leq w_2(v)$ pointwise, this equality enforces $f(v,t)=w_2(v)$
  for all $v\in V$.  We have therefore shown that $f$ has the desired 
  properties.

\end{proof}

\subsection{The case $N>2$}
This short section just shows how to iterate the $N=2$ case to complete
the proof of Proposition~\ref{discrete-decomp-propn}.

\begin{proof}[Proof of Proposition~\ref{discrete-decomp-propn}]
  We start with a simple observation. By~\eqref{local-cons-bd}, 
  since $N^+(V) \subset V$,
  \[
    \sum_{v\in V} w_1(v) \leq \sum_{v\in V} w_2(v)
    \leq \cdots \leq \sum_{v\in V} w_N(v).
  \]
  But by~\eqref{global-cons-bd} all of these inequalities are actually 
  equalities.  

  Now we consider the graph on the vertices
  $V_{[N]} := V\times [N]$, where $(u,i)\in V_{[N]}$ is connected to 
  $(v,i+1)\in V_{[N]}$ if $v\in N^+(u)$.  
  Each edge can be uniquely assigned a triple
  $(i,u,v)$ where $1\le i<N$ and $v\in N^+(u)$.  Denote by $E_{[N-1]}$ the 
  set of such triples.  We write
  $V_{[N]} = V_1\cup V_2\cup\cdots\cup V_N$ where $V_i = V\times\{i\}$.
  Similarly we write
  $E_{[N-1]} = E_1\cup E_2\cup\cdots\cup E_{N-1}$, where each $E_i$
  has the first coordinate of its triple equal to $i$.
  We would like to assign a weighting function $f:E_{[N-1]}\to \Real^+$ 
  to each of these edges such that for each $1\le i<N$ and $v\in V$,
  \begin{equation}
    w_i(u) = \sum_{v\in N^+(u)} f(i,u,v),
    \label{sum-out-eq}
  \end{equation}
  and for each $1<i\le N$,
  \begin{equation}
    w_i(v) = \sum_{u \in N^-(v)} f(i-1,u,v).
    \label{sum-in-eq}
  \end{equation}
  This can be done by finding, for each $1\le j<N$, a flow $f_j:E_i\to\Real^+$
  such that~\eqref{sum-out-eq} is satisfied for $i=j$ and~\eqref{sum-in-eq}
  is satisfied for $i=j+1$. But the existence of $f_j$ is exactly what is
  guaranteed by Lemma~\ref{one-layer-lemma}.  Now simply define $f$ on 
  $E_{[N-1]}$ by $\left.f\right|_{E_j} = f_j$.  
  
  Finally we use $f$ to define a weight on each path.  We approach this with
  a probabilistic interpretation.  Divide all weights $w_i$ and the flow $f$
  by the normalizing factor $\sum_{v\in V} w_1(v)$ to obtain new weights
  $\tilde{w}_i$ and a new flow $\tilde{f}$ which still obeys~\eqref{sum-out-eq}
  and~\eqref{sum-in-eq}.  Now consider the following random process.  I let
  $\nu_1\in V$ be a random variable with probability distribution given by 
  $\tilde{w}_1$.  Given $\nu_k\in V$, I independently choose a 
  $\nu_{k+1}\in N(\nu_k)$
  with probability distribution given by
  \[
    \Prob(\nu_{k+1} = v | \nu_k) = f(k, \nu_k, v) / w_k(\nu_k).
  \]
  Notice that for this probability to make sense we would need
  \[
    1 = \Prob(\nu_{k+1}\in N(\nu_k) | \nu_k)
    = \frac{1}{w_k(\nu_k)}\sum_{v\in N(\nu_k)} f(k,\nu_k,v),
  \]
  which is given by~\eqref{sum-out-eq}.  Once 
  $\nu_i$ is chosen for all $i\in [N]$, we have a randomly constructed path 
  $\rho\in\mathcal{P}_N$.  Let $\tilde{\alpha}(p) := \Prob(\rho=p)$ be the 
  probability distribution of $p$.  

  Now observe that for any $v\in V$, by the construction of $\rho$ and
  $\tilde{\alpha}$,
  \[
    \tilde{w}_1(v) = \Prob(\nu_1 = v) = \Prob(\rho(1) = v) = 
    \sum_{p(1)=v}\tilde{\alpha}(p),
  \]
  which up to the normalizing factor of $\sum_{v\in V} w_1(v)$ 
  establishes~\eqref{discrete-decomp-eq} for $i=1$.  
  For the sake of induction, suppose~\eqref{discrete-decomp-eq} is true for
  some $1\le i<N$.  Then for any $v\in V$,
  \begin{align*}
    \Prob(\nu_{i+1} = v) &= 
    \sum_{u\in N^-(v)}\Prob(\nu_{i}=u\text{ and }\nu_{i+1}=v) \\&= 
    \sum_{u\in N^-(v)}\Prob(\nu_i=u)\Prob(\nu_{i+1}=v | \nu_i=u)  \\&=
    \sum_{u\in N^-(v)}\tilde{w}_i(u) (\tilde{f}(i,u,v) / \tilde{w}_i(u)) 
    \\&= \sum_{u\in N(v)} \tilde{f}(i,u,v) = \tilde{w}_{i+1}(v)
  \end{align*}
  by~\eqref{sum-in-eq}.  But of course we can also write
  \[
    \Prob(\nu_{i+1} = v) = \Prob(\rho(i+1) = v)
    = \sum_{p(i+1)=v} \tilde{\alpha}(p).
  \]
  Thus by induction, and undoing the normalization, we have constructed our 
  desired $\alpha$.  
\end{proof} 

\section{The Skinny Lipschitz Tube Decomposition}
\label{decomp-section}
In this section we demonstrate how to combine the bounds from 
Section~\ref{ests-section} and the discrete problem from 
Section~\ref{discrete-decomp-section} to prove Theorem~\ref{tube-decomp-thm}.

\begin{proof}[Proof of Theorem~\ref{tube-decomp-thm}]
Let $u_t$ be a solution to the Schrodinger equation with frequency
localization~\eqref{freq-local-eq}.  Let $\mu$ be given by 
Proposition~\ref{LC-propn}, $H$ defined as in~\eqref{H-defn-eq}, 
and $\tau>0$ given by Proposition~\ref{FS-propn}.
We will discretize to the lattice 
$\Integer^d\times(\tau\Integer)\subset\Real^d\times\Real$; for convenience we
will reserve the letters $a,b$ for general elements of $\Integer^d$,
$h$ will always denote an element of $H$, and $x,y$ will always be used
for generic elements of $\Real^d$.  

We define for $n\in\Integer$ the weight function 
$m:\Integer^d\times (\tau\Integer)\to\Real^+$ by
\[
  m(a, n\tau) = \int |u_{n\tau}|^2 \mu_{a+H}.
\]
First observe that $m$ satisfies a conservation law since the translates 
of $\mu$ form a partition of unity:
\begin{equation}
  \sum_{a\in\Integer^d} m(a,n\tau) 
  = \int |u_{n\tau}|^2 \sum_{a\in\Integer^d} \mu_{a+H}
  = |H| \int |u_{n\tau}|^2 \mu_{\Integer^d} = |H| \int |u_{n\tau}|^2
  = |H| \|u_0\|_{L^2}^2.
  \label{cons-m-eq}
\end{equation}
Now we would like to show that $m$ also inherits a finite speed 
property from~\eqref{FS-bd}.  
Let $A\subset \Integer^d$ and let $n\in\Integer$.  We would like to show that 
\begin{equation}
  \sum_{a\in A} m(a,n\tau) \leq \sum_{b\in A+H} m(b, (n+1)\tau).
  \label{discrete-FS-bd}
\end{equation}
Expand the LHS using the definition of $m$:
\begin{align*}
  \sum_{a\in A} m(a, n\tau) 
  &= \int |u_{n\tau}|^2 \sum_{a\in A} \mu_{a+H} \\
  &= \int |u_{n\tau}|^2 \sum_{a\in A} \sum_{h\in H} \mu_{a+h}.
\end{align*}
A particular $a'\in A+H$ may contribute several times in the above integral.
Let 
\[
  E_{\ge i} = \{a'\in A+H; |A\cap (a'-H)| \ge i\},
\]
so that in other words $a'\in E_{\ge i}$ means there exists at least 
$i$ distinct choices of $h\in H$ such that $a'-h\in A$.  
Then we can write the above sum a different way, obtaining
\[
  \sum_{a\in A} m(a, n\tau) = 
  \sum_{i=1}^{|H|} \int |u_{n\tau}|^2 \mu_{E_{\ge i}}
\]
Now apply the finite speed bound~\eqref{FS-bd},
\[
  \sum_{a\in A} m(a, n\tau) \le
  \sum_{i=1}^{|H|} \int |u_{(n+1)\tau}|^2 \mu_{E_{\ge i}+H}.
\]
Similarly, if we define
\[
  F_{\ge i} = \{b'\in A+2H; |(A+H)\cap (b'-H)| \geq i\},
\]
then we can write
\[
  \sum_{b\in A+H} m(a, (n+1)\tau) =
  \sum_{i=1}^{|H|} \int |u_{(n+1)\tau}|^2 \mu_{F_{\ge i}}.
\]
Thus, to prove~\eqref{discrete-FS-bd} it would suffice to show
$E_{\ge i} + H \subset F_{\ge i}$.  Indeed, if 
$b\in E_{\ge i} + H$ then $b = a' + h$ for some $a'\in E_{\ge i}$ and 
$h\in H$, so by the definition of $E_{\ge i}$,
\begin{align*}
  i&\leq |\{(a,h')\in A\times H; \,\,a' = a + h'\}| \\
  &= |\{(a,h')\in A\times H; \,\,b = (a+h) + h'\}| \\
  &\leq |\{(c,h')\in (A+H)\times H; \,\,b = c+h'\}|
\end{align*}
This last inequality implies $b\in F_{\ge i}$.  We have therefore
shown~\eqref{discrete-FS-bd}.

We are nearing a situation where we can apply 
Proposition~\ref{discrete-decomp-propn}.  Let $R>0$, and let 
$N = \lceil 2R/\tau\rceil$ (rounded to the next even number for convenience).  
We consider the graph on $G = (\Integer^d, E)$
on $\Integer^d$ where $(a,b)\in E$ if and only if $a\in b+H$.  Define the
sequence of weights $w_i:\Integer^d\to\Real^+$ for $1\le i\le N$ by
\[
  w_i(a) = m(a,(i-N/2)\tau).
\]
Now~\eqref{discrete-FS-bd} implies~\eqref{local-cons-bd}, and the conservation
of mass from~\eqref{cons-m-eq} implies~\eqref{global-cons-bd}.  Thus 
by Proposition~\ref{discrete-decomp-propn} there 
is a weight $\alpha:\mathcal{P}_N\to\Real^+$ on the set of paths of length
$N$ on $G$ such that 
\[
  m(a, (i-N/2)\tau) = \sum_{p(i)=a} \alpha(p).
\]
Let $p$ be one such path.  Choose $r=10d$, and let $\gamma_p$ be the piecewise
linear path with $\gamma_p( (i-N/2)\tau) = p(i)$.  By taking $V=C_d/\tau$,
we obtain the velocity bound $|\nabla \gamma_p|\leq V$.
Consider the tube decomposition function
\[
  f(x,t) = \sum_{p\in\mathcal{P}_N} \alpha(p) T_{\gamma_p, r}(x,t).
\]
Now we verify that $f(x,t)$ is a pointwise upper bound for $|u_t(x)|^2$.  
Let $a\in \Integer^d$ and $1\le i\le N$.  Suppose that 
$(x,t)\in\Real^d\times\Real$ is 
contained in the prism with bottom-left corner $(a, (i-N/2)\tau)$, meaning that
\[
  a_j\leq x_j < a_j+1, \quad (i-N/2) \tau \leq t<(i-N/2+1)\tau
\]
for $1\le j\le d$. Now let $p\in\mathcal{P}_N$ be some path
with $p(i) = a$.  Then because we chose $r$ large enough,
and because $\gamma_p( (i-N/2) ) = p(i) = a$, it follows that 
$(x,t)\in T_{\gamma_p,r}$.  Using this observation we can make the bound
\[
  f(x,t) = \sum_{p\in\mathcal{P}_N}\alpha(p) T_{\gamma_p,r}(x,t)
  \geq \sum_{p(i) = a} \alpha(p)  = m(a, (i-N/2)\tau).  
\]
Now by~\eqref{LC-bd} and~\eqref{FS-bd}, there exists some
absolute constant $C$ such that 
\begin{equation}
  |u_t(x)|^2 \leq C\int |u_t|^2 \mu_a \leq C\int |u_{(i-N/2)\tau}|^2\mu_{a+H}
  = Cm(a,n\tau). 
  \label{m-dominated}
\end{equation}
Thus for every $(x,t)$ with $|t|<R$, 
\[
  |u_t(x)|^2 \leq C\sum_{p\in\mathcal{P}_N} \alpha(p) T_{\gamma_p,r}(x,t),
\]
which is exactly (with slightly different letters) the bound~\eqref{ptwise-bd}.

Now we verify~\eqref{L1-efficient-bd}.  Indeed, by~\eqref{cons-m-eq},
\[
  \sum_{p\in\mathcal{P}_N} (10d)^d \alpha(p)
  \leq C \sum_{a\in\Integer^d} w_1(a) = C|H|\|u_0\|_{L^2}^2.
\]
\end{proof}

%% file: bilinear.tex
\section{Applications to bilinear and multilinear estimates}
\label{applications-section}
In this section we prove the bilinear Strichartz estimate, 
Theorem~\ref{bi-Strich-thm} and the multilinear restriction theorem,
Theorem~\ref{multi-rest-thm} using the tube decomposition.  Neither result is
new, but I think that these proofs illustrate a few interesting points. 
Perhaps one of the main lessons is that neither the bilinear Strichartz
nor the multilinear restriction estimates are
truly \emph{dispersive}.  Indeed observe that the decomposition
described in Theorem~\ref{tube-decomp-thm} does not
rule out the possibility that there is only one tube.  Thus by itself it
cannot imply any dispersive estimates.

\subsection{Proof of Bilinear Strichartz}
\label{bilinear-section}
In this part we prove Theorem~\ref{bi-Strich-thm} using
the decomposition given by Theorem~\ref{tube-decomp-thm}.  I believe this
proof illustrates something intuitive about the estimate -- roughly speaking
the intuition is that a fast and a slow particle cannot be in the same place
for very long.  This intuition is expressed in the proof by observing that 
a ``fast'' and ``slow'' tube must have a small region of intersection.

Recall the definition of the annulus,
\[
  A_N := \{\xi;\, N/2\leq |\xi|\leq 2N\}.
\]
We define in addition the annulus
\[
  A^*_N := \{\xi;\, N/4 \leq |\xi|\leq 4N\},
\]
\begin{proof}[Proof of Theorem~\ref{bi-Strich-thm}]
  Let $u_0, v_0 \in L^2(\Real^d)$ have frequency localizations
  $\supp \hat{u_0} \subset A_N$ and $\supp \hat{v_0} \subset A_M$ 
  with $M\ll N$.  
  
  Let $V$ be the constant appearing in Theorem~\ref{tube-decomp-thm}
  (which is the same constant appearing in Corollary~\ref{scaled-decomp-cor}),
  and let $\{\xi_i\}_{i=1}^{(100V)^d} \subset A_1$ be a set of points such that 
  the balls with radius $1/(10V)$ 
  centered at $\{\xi_i\}$ cover $A_1$ and are contained in
  $A^*_1$:
  \[
    A_1 \subset \bigcup_{i=1}^{(100V)^d} B_{1/10V}(\xi_i) \subset A^*_1.
  \]
  Using a partition of unity on (a scaled version of) this 
  covering, write the decompositions 
  \[
    u_0 = \sum_{i=1}^{(100V)^d} u_i, \quad v_0 = \sum_{i=1}^{(100V)^d} v_i
  \]
  such that $\supp \hat{u_i} \subset B_{N/(10V)}(N\xi_i)$ and 
  $\supp \hat{v_i} \subset B_{M/(10V)}(M\xi_i)$.  We can also demand that 
  these decompositions are nearly orthogonal in the sense that 
  \[
    \sum_{i=1}^{(100V)^d} \|u_i\|_{L^2_x}^2 \leq 2 \|u\|_{L^2}^2
  \]
  and likewise with $v_i$.  Using this decomposition and the 
  Cauchy-Schwartz inequality, we obtain
  \[
    \|uv\|_{L^2_{t,x}(\Real^d\times\Real)}
    \leq (100V)^{2d} \sum_{i=1}^{(100V)^d}\sum_{j=1}^{(100V)^d}
    \|u_iv_j\|_{L^2_{t,x}(\Real^d\times\Real)}.
  \]
  Now we bound these cross-terms $\|u_i v_j\|_{L^2_{t,x}(\Real^d\times\Real)}$.
  From Corollary~\ref{scaled-decomp-cor} we obtain paths 
  $\{\gamma_n\}_{n=1}^\infty$ with weights $\{w_n\}_{n=1}^\infty$ such that
  \[
    |u_i(x,t)|^2 \leq \sum_{n=1}^\infty w_n T_{\gamma_n,rN^{-1}}(x,t).
  \]
  For convenience we will write $T_n$ instead of $T_{\gamma_n,rN^{-1}}$.  
  We also use Corollary~\ref{scaled-decomp-cor} on $v_j$ to obtain 
  paths $\{\gamma'_m\}_{m=1}^\infty$ with weights $\{w'_m\}_{m=1}^\infty$
  such that
  \[
    |v_j(x,t)|^2 \leq \sum_{m=1}^\infty w'_m T'_m(x,t)
  \]
  where $T'_m(x,t) = T_{\gamma'_m, rM^{-1}}$.  Applying these pointwise bounds,
  we estimate
  \begin{align}
    \begin{split}
    \int_{\Real}\int_{\Real^d} |u_i(x,t)|^2|v_j(x,t)|^2 \,dx\,dt
    \leq \sum_{m=1}^\infty \sum_{n=1}^\infty w_n w'_m
    \int_{\Real}\int_{\Real^d} T_n(x,t) T'_m(x,t)\,dx\,dt.
    \label{ptwise-decomp-bd}
    \end{split}
  \end{align}
  Now the paths $\gamma_n$ and $\gamma'_m$ defining $T_n$ and $T_m$ have 
  velocities satisfying 
  \[
    |\nabla \gamma_n - N\xi_i| \leq N/10
  \]
  and 
  \[
    |\nabla \gamma'_m - M\xi_j| \leq M/10.
  \]
  Since $\xi_i,\xi_j\in A_1$, we can conclude that 
  $\nabla \gamma_n\subset A^*_N$ and $\nabla \gamma'_m\subset A^*_M$.  Thus
  the tubes $T_n$ and $T'_m$ can meet at most once, in a region of spatial 
  width $N^{-1}$ and only for a duration of $(MN)^{-1}$.  That is,
  \[
    \int_{\Real}\int_{\Real^d} T_n(x,t) T'_m(x,t)\,dx\,dt
    \leq C N^{-1-d} M^{-1}.
  \]
  Plugging this back into~\eqref{ptwise-decomp-bd} and then 
  applying~\eqref{scaled-L1-bd}, 
  \begin{align*}
    \|u_iv_j\|_{L^2_{t,x}(\Real^d\times\Real)}^2 
    &\leq C \sum_{m=1}^\infty \sum_{n=1}^\infty N^{-1-d} M^{-1} w_n w'_m \\
    &\leq C N^{-1} M^{d-1} (\sum_{m=1}^\infty M^{-d} w'_m)
    (\sum_{n=1}^\infty N^{-d} w_n) \\
    &\leq C N^{-1} M^{d-1} \|u_i\|_{L^2_x}^2 \|v_j\|_{L^2_x}^2.
  \end{align*}
  Finally applying the almost orthogonality of the decomposition allows us 
  to conclude
  \[
    \sum_{i=1}^{(100V)^d} \sum_{j=1}^{(100V)^d}
    \|u_i v_j\|_{L^2_{t,x}(\Real^d\times\Real)}^2
    \leq C N^{-1} M^{d-1} 
    (\sum_{i=1}^{(100V)^d} \|u_i\|_{L^2_x}^2)
    (\sum_{j=1}^{(100V)^d} \|v_j\|_{L^2_x}^2)
    \leq 2C N^{-1} M^{d-1} \|u\|_{L^2_x}^2 \|v\|_{L^2_x}^2.
  \]
  The theorem is finished by taking square roots.
\end{proof}

\subsection{Multilinear restriction}
We can also prove the multilinear restriction estimate using the multilinear
Kakeya problem for Lipschitz tubes, proved by Guth in~\cite{guth2015short}.  
First let us set up some notation.  Let $\{v_i\}_{i=1}^n\subset\Real^n$ be 
a collection of unit vectors.  We say that $\{v_i\}$ is $\nu$-transverse
if $|v_1\wedge v_2\wedge\cdots\wedge v_n| \ge \nu$.
Let $\Gamma_i = \{\gamma_{i,j}\}_{j=1}^{N_j}$ be a collection of Lipschitz
curves.  Given $\delta>0$, we say that $\Gamma_i$ is $\delta$-close to
parallel with $v_i$ if for any tangent vector $\vec{t}$ of a curve 
$\gamma_{i,j}$, $|\vec{t} - v_j| < \delta$.  Let $T_{i,j}$ denote the tube
of unit radius centered along the curve $\gamma_{i,j}$.

\begin{thm}[Multilinear Kakeya for Lipschitz tubes~\cite{guth2015short}]
  For every $\eps>0$ and $\nu>0$, there is a $C_{\eps,\nu}>0$ and a 
  $\delta$ such that the folllowing
  holds:  If $\{v_i\}\subset\Real^n$ is a set of unit vectors that is
  $\nu$-transverse, and each $\Gamma_i = \{\gamma_{i,j}\}_{j=1}^{N_i}$ 
  is a set of Lipschitz paths which are $\delta$-close to $v_i$, and 
  $w_{i,j}$ is any positive collection of weights, then
  \begin{equation}
    \int_{B_R} \prod_{i=1}^n
    \left(\sum_{j=1}^{N_j} w_{i,j} T_{i,j}\right)^{\frac{1}{n-1}}
    \leq
    C_\eps R^\eps \prod_{i=1}^n \left(
    \sum_{j=1}^{N_i} w_{i,j}\right)^{\frac{1}{n-1}}.
    \label{mult-kakeya-bd}
  \end{equation}
  \label{mult-kakeya-thm}
\end{thm}

This theorem implies a version of multilinear restriction for solutions
to the Schr{\"o}dinger equation which are highly localized in frequency.
\begin{lemma}
  For every $\eps>0$ there exists a $C_\eps$ and a $\delta > 0$ such 
  that the following holds: 
  Let $r=(10d)^{-1}$ and suppose
  $\{\xi_i\}_{i=0}^d$ is a set of vectors with $\xi_0\in B_r(0)$ 
  and $\xi_i \in B_r(e_i)$ for $1\le i\le d$.  If also $u_i$ are solutions
  to the Schrodinger equation with $\supp \hat{u}_i \in B_\delta(\xi_i)$,
  then 
  \[
    \int_{|x|<R} \int_{|t|<R} 
    \prod_{i=0}^d |u_i(t,x)|^{2/d}\,dx\,dt
    \leq C_\eps R^\eps \prod_{i=0}^d \|u_i\|_{L^2_x}^{2/d}
  \]
  for all $R>0$.
\end{lemma}
\begin{proof}
  Consider the vectors $v'_i = (1, -2\xi'_i) \in \Real^{d+1}$, and let
  $v_i = v'_i / |v'_i|$.  The vectors $\{v_i\}$ correspond to the unit normals
  to the paraboloid $P = \{\tau = |\xi|^2\} \subset \Real^{d+1}$ at the 
  points $(|\xi_i|^2, \xi_i)$.  Because of the constraint on the positions
  of $\xi_i$, we can guarantee that the set $\{v_i\}$ is $\nu$-transverse
  for some absolute constant $\nu>0$.  
  
  Let $\eps>0$, and let $C_\eps$ and $\tilde{\delta}$ be given
  according to Theorem~\ref{mult-kakeya-thm} with this constant $\nu$
  and with $n=d+1$.  Let
  $V$ be the same constant that appears in Theorem~\ref{tube-decomp-thm},
  and define $\delta = \tilde{\delta} / V$.   

  Now if $\supp \hat{u}_i \subset B_{\delta} (\xi_i)$, then according to
  Corollary~\ref{scaled-decomp-cor}, it is possible to find paths 
  $\{\gamma_{i,j}\}_{j=1}^\infty$ such that 
  $|\nabla \gamma_{i,j}-\xi_i|<\tilde{\delta}$ and weights $\{w_{i,j}\}$
  such that 
  \[
    |u_i(t,x)|^2 \leq \sum_j w_{i,j} T_{\gamma_{i,j}, r\delta^{-1}}(t,x)
  \]
  and 
  \[
    \sum_j (r\delta^{-1})^d w_{i,j} \leq C \|u_i\|_{L^2_x}^2.
  \]
  From now on we will simply write $T_{i,j}$ for the tube 
  $T_{\gamma_i,r\delta^{-1}}$ centered at 
  $\gamma_{i,j}$ with width $r\delta^{-1}$.  Thinking of $\gamma_{i,j}$ as
  a curve in $\Real^{d+1}$, where the first coordinate is time and the remaining
  $d$ coordinates are spatial, the condition 
  $|\nabla \gamma_{i,j}-\xi_i|<\tilde{\delta}$ ensures that the curves 
  $\{\gamma_{i,j}\}$ are $\tilde{\delta}$-close to $v_i$.  Since 
  the radius of the tubes $T_{i,j}$ is $r\delta^{-1}$, which is much larger
  than $1$, we may by splitting the tubes into thinner pieces think of them
  as having radius $1$ (buying simplicity at the cost of an ignorable 
  factor of $\delta^{\eps}$).  Thus we may apply Theorem~\ref{mult-kakeya-thm}
  to obtain
  \begin{align*}
    \int_{|x|<R/2}\int_{|t|<R/2}
    \prod_{i=0}^d |u_i(t,x)|^{2/d}
    &\leq \int_{B_R} 
    \prod_{i=0}^d \left(\sum_{j=1}^\infty w_{i,j} T_{i,j}\right)^{1/d}
    \\&\leq C_\eps R^\eps \prod_{i=0}^d 
    \left(\sum_{j=1}^{\infty} w_{i,j}\right)^{1/d}
    \\&\leq C_\eps R^\eps \prod_{i=0}^d \|u_i\|_{L^2_x}^{2/d}.
  \end{align*}
\end{proof}

The multilinear restriction theorem, Theorem~\ref{multi-rest-thm}, follows
upon decomposing each of the $d+1$ solutions into finitely many pieces with
frequency localizations given by $\delta$, and then summing the finitely
many contributions.

%% file: main.bbl
\newcommand{\etalchar}[1]{$^{#1}$}
\providecommand{\bysame}{\leavevmode\hbox to3em{\hrulefill}\thinspace}
\providecommand{\MR}{\relax\ifhmode\unskip\space\fi MR }
\providecommand{\MRhref}[2]{%
  \href{http://www.ams.org/mathscinet-getitem?mr=#1}{#2}
}
\providecommand{\href}[2]{#2}
\begin{thebibliography}{CKS{\etalchar{+}}02}

\bibitem[BCT06]{BCT06}
J~Bennett, A~Carbery, and T~Tao, \emph{On the multilinear restriction and
  {K}akeya conjectures}, Acta mathematica \textbf{196} (2006), no.~2, 261--302.

\bibitem[BD14]{BD14}
J~Bourgain and C~Demeter, \emph{The proof of the $\ell^2$ decoupling
  conjecture}, arXiv preprint arXiv:1403.5335 (2014).

\bibitem[BG11]{BG11}
J~Bourgain and L~Guth, \emph{Bounds on oscillatory integral operators based on
  multilinear estimates}, Geometric and Functional Analysis \textbf{21} (2011),
  no.~6, 1239--1295.

\bibitem[Bou98]{bourgain98}
J.~Bourgain, \emph{Refinements of {S}trichartz' inequality and applications to
  2{D}-{NLS} with critical nonlinearity}, Intern. Mat. Res. Notices \textbf{5}
  (1998), 253--283.

\bibitem[CKS{\etalchar{+}}02]{CKSTT02}
J.~Colliander, M.~Keel, G.~Staffilani, H.~Takaoka, and T.~Tao, \emph{Almost
  conservations laws and global rough sol.utions to a nonlinear
  {S}chr{\"o}dinger equation}, Math. Res. Letters \textbf{9} (2002), 659--682.

\bibitem[FF56]{ford1956maximal}
LR~Ford and DR~Fulkerson, \emph{Maximal flow through a network}, Canadian
  journal of Mathematics \textbf{8} (1956), no.~3, 399--404.

\bibitem[Gut15]{guth2015short}
L~Guth, \emph{A short proof of the multilinear {K}akeya inequality},
  Mathematical Proceedings of the Cambridge Philosophical Society, vol. 158,
  Cambridge Univ Press, 2015, pp.~147--153.

\bibitem[Han12a]{hani2012}
Z.~Hani, \emph{A bilinear oscillatory integral estimate and bilinear
  refinements to {S}trichartz estimates on closed manifolds}, Analysis and PDE
  \textbf{5} (2012), 339--362.

\bibitem[Han12b]{hani2012global}
Z.~Hani, \emph{Global well-posedness of the cubic nonlinear {S}chr{\"o}dinger
  equation on closed manifolds}, Communications in Partial Differential
  Equations \textbf{37} (2012), no.~7, 1186--1236.

\bibitem[Jos07]{Joseph07}
S.~Joseph, \emph{The max-flow min-cut theorem},
  \url{http://www.math.uri.edu/~eaton/maxflowmincut.pdf}, 2007.

\bibitem[KRT02]{KRT02}
S.~Klainerman, I.~Rodnianski, and T.~Tao, \emph{A physical space approach to
  wave equation bilinear estimates}, Journal d’Analyse Math{\'e}matique
  \textbf{87} (2002), no.~1, 299--336.

\bibitem[SM]{Staples-Moore}
A~Staples-Moore, \emph{Network flows and the max-flow min-cut theorem},
  \url{http://www.math.uchicago.edu/~may/VIGRE/VIGRE2009/REUPapers/Staples-Moore.pdf}.

\bibitem[Tao10]{tao2010}
T.~Tao, \emph{A physical space proof of the bilinear {S}trichartz and local
  smoothing estimate for the {S}chr{\"o}dinger equation},
  \url{http://terrytao.files.wordpress.com/2013/07/bilinear.pdf}, 2010.

\end{thebibliography}
